\documentclass[10pt]{amsart}
\usepackage{amsmath, amssymb, amsthm, bbold}
\usepackage{paralist, xcolor, tikz, hyperref}
\usepackage[margin=1in,letterpaper,portrait]{geometry}
\usepackage{enumitem}
\usepackage{caption}
\usepackage[noadjust]{cite}
\usepackage{comment}
\usepackage{tikz, tikz-cd} 
 \usetikzlibrary{braids,arrows.meta}
\usepackage[normalem]{ulem}
\hypersetup{
    colorlinks=true,
    urlcolor=blue}

\theoremstyle{plain}
\newtheorem{conjecture}{Conjecture}[]
\newtheorem{theorem}[conjecture]{Theorem}

\newtheorem{lemma}[conjecture]{Lemma}
\newtheorem{proposition}[conjecture]{Proposition}

\newtheorem{remark}[conjecture]{Remark}

\newtheorem*{question*}{Question}

\newcommand{\zero}{\mathbb{0}}
\newcommand{\one}{\mathbb{1}}
\newcommand{\ee}{\mathbf{e}}

\newcommand{\ZZ}{\mathbb{Z}}
\newcommand{\RR}{\mathbb{R}}

\newcommand{\spn}{\mathrm{span}}

\newcommand{\sign}{\mathrm{sign}}

\newcommand{\colorize}[1]{{\textcolor{red}{#1}}}

\author{Jang Soo Kim} 
\address{Department of Mathematics, Sungkyunkwan University (SKKU), Suwon, Gyeonggi-do 16419, South Korea}
\email{jangsookim@skku.edu}

\author{Victor Reiner}
\address{School of Mathematics, University of Minnesota, Minneapolis MN, USA}
\email{reiner@umn.edu}

\author{Laura Schoenzeit}   
\address{School of Mathematics, University of Minnesota, Minneapolis MN, USA}
\email{lschoen3@ncsu.edu}


\title{Note on adjacency spectra for twice punctured hypercubes}

\begin{document}

\begin{abstract}
Eigenspaces of hypercube graph adjacency matrices are well understood. We extend this to subgraphs obtained by removing a vertex or antipodal pair of vertices, motivated by a question on zonotopes.
\end{abstract}

\maketitle

%%%%%%%%%%%%%%%%%%%%%%%%%%%%%%%%%%%%
\section{Introduction}
\label{sec:intro}
%%%%%%%%%%%%%%%%%%%%%%%%%%%%%%%%%%%%

This note addresses a conjecture \cite[Conj.~3.3]{schoenzeit} arising from work of the third author
on spectra of adjacency matrices for the graphs of vertices and edges of certain {\it zonotopes}. A zonotope is the Minkowski sum
$
Z=\{\sum_{i=1}^n c_i v_i: c_i \in [0,1]\}
$
of a collection of line segments in the directions of vectors $v_1,\ldots,v_n$ in $\RR^d$;  equivalently, $Z$ is the image of the {\it $n$-cube} $[0,1]^n$ under a linear map $\RR^n \rightarrow \RR^d$.
An interesting family of zonotopes, 
denoted $Z_{n-1}$ in \cite{schoenzeit},
arises when $\{v_i\}_{i=1}^n$ form a set of $n$ {\it minimally dependent} vectors in $\RR^{n-1}$,
so that all of its proper subsets are linearly independent.  Without altering the facial structure of $Z_{n-1}$, one can rescale each vector $v_i$, change coordinates,
and identify $\RR^{n-1}$ with the hyperplane $\{x\in\RR^n:\sum_{i=1}^n x_i=0\}$ so as to make
$$
\{v_1,v_2,\ldots,v_{n-1},v_n\}=\{e_1-e_2,e_2-e_3,\ldots,e_{n-1}-e_n,e_n-e_1\},
$$
where $e_1,\ldots,e_n$ are the standard basis vectors in $\RR^n$.  Thus $Z_{n-1}$ is the {\it graphical zonotope} for the $n$-cycle graph, as studied by Escobar and McWhirter \cite{escobarmcwhirter} and
Gruji\'c \cite{Grujic}; see \cite[\S2]{Grujic} for more on graphic zonotopes.
It is known (see\footnote{This is the only reference that we found, but earlier references are welcome!} \cite[Cor.~3.7]{escobarmcwhirter}) that the
boundary complex of $Z_{n-1}$ is isomorphic to the subcomplex of faces of
the $n$-cube containing neither of the antipodal vertices \(\zero_n:=(0,\dots,0)\) and \( \one_n:=(1,\dots,1) \). Thus the graph of $Z_{n-1}$
is obtained from the $n$-cube graph $Q_n$ as the deletion $\hat{\hat{Q}}_n$ of the antipodal pair of vertices $\zero_n, \one_n$.  Pictured below are $Q_3, Q_4$ and $\hat{\hat{Q}}_3,\hat{\hat{Q}}_4$, with vertices labeled by subsets, along with the graphs of the {\it hexagon} $Z_2$ and {\it rhombic dodecahedron} $Z_3$:

%%%%% n=3 pictures
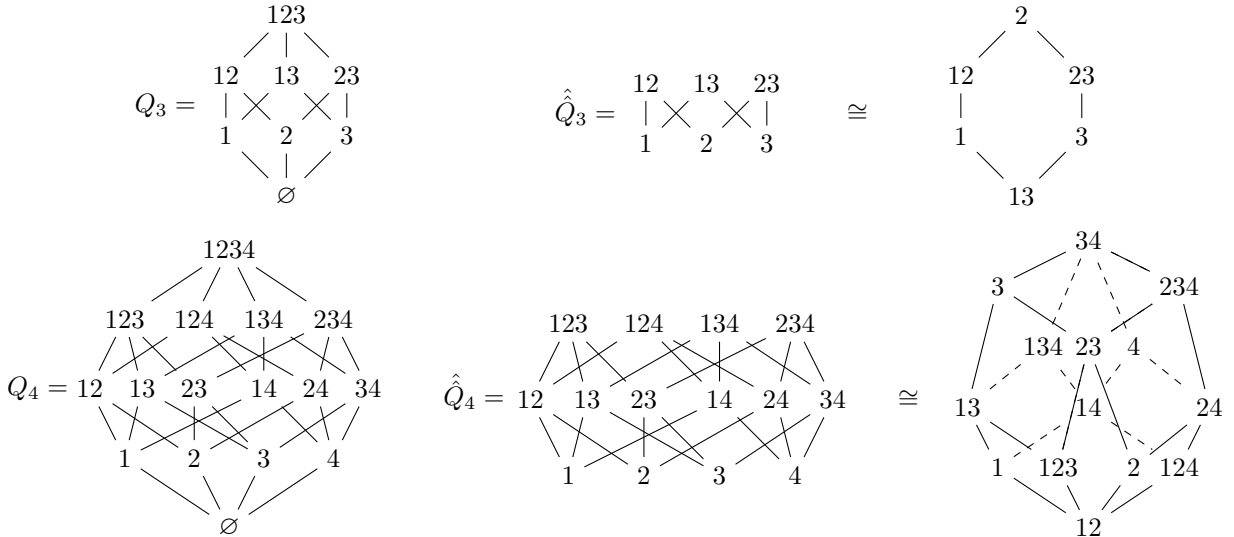
\begin{center}
\begin{tikzpicture}
  [scale=.2,auto=left]
  \node (empty) at (0,0) {$\varnothing$}; 
   \node (1) at (-4,4) {$1$}; 
   \node (2) at (0,4) {$2$}; 
   \node (3) at (4,4) {$3$}; 
   \node (12) at (-4,8) {$12$}; 
   \node (13) at (0,8) {$13$}; 
   \node (23) at (4,8) {$23$}; 
   \node (123) at (0,12) {$123$}; 
   \node (Q3) at (-8,6) {$Q_3=$};
  \foreach \from/\to in {empty/1,empty/2,empty/3,
1/12,1/13,2/12,2/23,3/13,3/23,12/123,13/123,23/123}
    \draw (\from) -- (\to);
\end{tikzpicture}
\qquad \qquad \qquad
\raisebox{0.7cm}{
\begin{tikzpicture}
  [scale=.2,auto=left]
   \node (1) at (-4,8) {$1$}; 
   \node (2) at (0,8) {$2$}; 
   \node (3) at (4,8) {$3$}; 
   \node (12) at (-4,12) {$12$}; 
   \node (13) at (0,12) {$13$}; 
   \node (23) at (4,12) {$23$}; 
   \node (label) at (-8,10.4) {$\hat{\hat{Q}}_3=$};
   \node (cong) at (10,10) {$\cong$};
  \foreach \from/\to in {
1/12,1/13,2/12,2/23,3/13,3/23}
    \draw (\from) -- (\to);
\end{tikzpicture}}
\qquad 
\begin{tikzpicture}
  [scale=.2,auto=left]
   \node (1) at (-4,0) {$1$}; 
   \node (2) at (0,8) {$2$}; 
   \node (3) at (4,0) {$3$}; 
   \node (12) at (-4,4) {$12$}; 
   \node (13) at (0,-4) {$13$}; 
   \node (23) at (4,4) {$23$}; 
  \foreach \from/\to in {
1/12,1/13,2/12,2/23,3/13,3/23}
    \draw (\from) -- (\to);
\end{tikzpicture}
\end{center}

%\vskip.1in
%%%%% n=4 pictures

\begin{center}
\begin{tikzpicture}
  [scale=.23,auto=left]
  \node (empty) at (0,0) {$\varnothing$}; 
   \node (1) at (-6,4) {$1$}; 
   \node (2) at (-2,4) {$2$}; 
   \node (3) at (2,4) {$3$}; 
   \node (4) at (6,4) {$4$}; 
   \node (12) at (-8,8) {$12$}; 
   \node (13) at (-5,8) {$13$}; 
   \node (23) at (-2,8) {$23$}; 
   \node (14) at (2,8) {$14$}; 
   \node (24) at (5,8) {$24$}; 
   \node (34) at (8,8) {$34$}; 
   \node (123) at (-6,12) {$123$}; 
   \node (124) at (-2,12) {$124$}; 
   \node (134) at (2,12) {$134$}; 
   \node (234) at (6,12) {$234$}; 
   \node (1234) at (0,16) {$1234$}; 
   \node (Q_4) at (-11,8) {$Q_4=$};
  \foreach \from/\to in {empty/1,empty/2,empty/3,empty/4,
1/12,1/13,1/14,
2/12,2/23,2/24,
3/13,3/23,3/34,
4/14,4/24,4/34,
12/123,12/124,13/123,13/134,23/123,23/234,14/124,14/134,24/124,24/234,34/134,34/234,
123/1234,124/1234,134/1234,234/1234}
    \draw (\from) -- (\to);
\end{tikzpicture}
\quad 
\raisebox{0.7cm}{
\begin{tikzpicture}
  [scale=.25,auto=left]
   \node (1) at (-6,4) {$1$}; 
   \node (2) at (-2,4) {$2$}; 
   \node (3) at (2,4) {$3$}; 
   \node (4) at (6,4) {$4$}; 
   \node (12) at (-8,8) {$12$}; 
   \node (13) at (-5,8) {$13$}; 
   \node (23) at (-2,8) {$23$}; 
   \node (14) at (2,8) {$14$}; 
   \node (24) at (5,8) {$24$}; 
   \node (34) at (8,8) {$34$}; 
   \node (123) at (-6,12) {$123$}; 
   \node (124) at (-2,12) {$124$}; 
   \node (134) at (2,12) {$134$}; 
   \node (234) at (6,12) {$234$}; 
   \node (label) at (-11,8.4) {$\hat{\hat{Q}}_4=$};
  \node (cong) at (12,8) {$\cong$};
\foreach \from/\to in {1/12,1/13,1/14,
2/12,2/23,2/24,
3/13,3/23,3/34,
4/14,4/24,4/34,
12/123,12/124,13/123,13/134,23/123,23/234,14/124,14/134,24/124,24/234,34/134,34/234}
\draw (\from) -- (\to);
\end{tikzpicture}
}
\begin{tikzpicture}
  [scale=.20,auto=left]
   \node (1) at (-6,4) {$1$}; 
   \node (2) at (3,4) {$2$}; 
   \node (3) at (-6,16) {$3$}; 
   \node (4) at (3,12) {$4$}; 
   \node (12) at (0,0) {$12$}; 
   \node (13) at (-8,8) {$13$}; 
   \node (23) at (0,12) {$23$}; 
   \node (14) at (0,8) {$14$}; 
   \node (24) at (8,8) {$24$}; 
   \node (34) at (0,19) {$34$}; 
   \node (123) at (-2,4) {$123$}; 
   \node (124) at (6,4) {$124$}; 
   \node (134) at (-3,12) {$134$}; 
   \node (234) at (6,16) {$234$};
     \foreach \from/\to in {1/12,1/13,
2/12,2/23,2/24,
3/13,3/23,3/34,
12/123,12/124,
13/123,
23/123,23/234,
24/124,24/234,
34/234}
    \draw(\from) -- (\to);
    \foreach \from/\to in {1/14,
4/14,4/24,4/34,
13/134,
23/123,23/234,
14/124,14/134,
34/134,34/234}
    \draw[dashed] (\from) -- (\to);
\end{tikzpicture}
\end{center}

We will exhibit a simple relation between the spectra of the $\{0,1\}$-adjacency matrices for $Q_n, \hat{\hat{Q}}_n$, and extending also to the intermediate graph $\hat{Q}_n$ obtained by deleting only the vertex $\zero_n$ from $Q_n$.  Our goal is Theorem~\ref{thm:main} below,
a precise description of the  eigenspaces for the adjacency matrices  $A_{Q_n}, A_{\hat{Q}_n}, A_{\hat{\hat{Q}}_n}$, including the action of the symmetric group $S_n$ on the eigenspaces.
In particular, this proves \cite[Conj.~3.3]{schoenzeit}. 

%%%%%%%%%%%%%%%%%%%%%%%%%%%%%%%%%%%%%%%%%%%%%%%%%
\section{The hypercube and its adjacency eigenspaces}
\label{sec:hypercube-review}
%%%%%%%%%%%%%%%%%%%%%%%%%%%%%%%%%%%%%%%%%%%%%%%%%

Label the vertices of
the hypercube $[0,1]^n$
as $u=(u_1,\ldots,u_n)$ in $\ZZ_2^n$ where 
$
\ZZ_2:=\ZZ/2\ZZ=\{0,1\}.
$
Two vertices $u,v$ of its graph $Q_n$ will be adjacent exactly when the {\it Hamming distance}
$d(u,v):=\#\{i:u_i \neq v_i\}$ is $1$.
Define also the {\it Hamming weight}
$\omega(u):=\#\{i:u_i=1\}=d(u,\zero_n)$
for $u \in \ZZ_2^n$.

We will view the adjacency matrix $A_{Q_n}$ as a linear operator on the space $V=\RR^{\ZZ_2^n}$. 
More precisely, letting \( \{\ee_u: u\in \ZZ_2^n\} \) be
  the standard basis of \( V \), we have
  \begin{equation}\label{eq:AQn}
    A_{Q_n} \ee_u = \sum_{\substack{v \in \ZZ_2^n:\\ d(u,v)=1}} \ee_v.
  \end{equation}
The symmetric group $S_n$ permutes the set $\ZZ_2^n$
  by \( \sigma(u_1,\dots,u_n) = (u_{\sigma^{-1}(1)},\dots,u_{\sigma^{-1}(n)}) \).
  This $S_n$-action is via  graph symmetries on $Q_n$,
  and thus when $S_n$ permutes the coordinates in $V=\RR^{\ZZ_2^n}$, it will preserve the eigenspaces for the adjacency matrices $A_{Q_n}$.

To describe eigenvectors for the adjacency
matrix $A_{Q_n}$, introduce a $\ZZ_2$-valued dot product on $\ZZ_2^n$ by
$u\cdot v:=\sum_{i=1}^n u_i v_i$.
Then for $u \in \ZZ_2^n$, define a vector $E_u \in \RR^{\ZZ_2^n}$ with coordinates 
$(E_u)_v:=(-1)^{u \cdot v}$.
Note that a permutation $\sigma$ in $S_n$ will have
$\sigma(E_u)=E_{\sigma(u)}$. 
Define a subspace 
\begin{equation}
\label{eq:hypercube-eigenspaces}
V_{n-2k} := \spn_\RR\{E_u: u \in \ZZ_2^n, \omega(u)=k\}.
\end{equation}
Let $(-,-)$ be the usual inner product on  $V=\RR^{\ZZ_2^n}$ for which
$\{\ee_u: u\in \ZZ_2^n\}$
are an orthonormal basis.
The next proposition is standard;  see Stanley \cite[Chapter~2]{Stanley-undergrad-text},   
Cvetkovic, Doob, and Sachs \cite[\S2.6, Example~10]{CvetkovicDoobSachs}.

\begin{proposition}
\label{prop:eigenspaces-for-hypercube}
The vectors \(\{E_u\}_{u \in \ZZ_2^n}\) form an orthogonal basis of
\(V=\RR^{\ZZ_2^n}\) with respect to $(-,-)$:
one has $(E_u,E_v) = 2^n$ if $u=v$ and $(E_u,E_v) =0$ otherwise.
%\[  E_u\cdot E_v =  \begin{cases}   2^n & \mbox{if \( u=v \)},\\   0 & \mbox{otherwise.}  \end{cases} \]
They are eigenvectors for \(A_{Q_n}\) with the following eigenvalue
equations: for \( u\in \ZZ_2^n \) with \( k=\omega(u) \),
\begin{equation}
\label{eqn:eigenvector-equation-for-Eu}
A_{Q_n}(E_u)=(n-2k) E_u.
\end{equation}
Thus for each \(k=0,1,\ldots,n\), the matrix
\(A_{Q_n}\) has eigenvalue \(n-2k\) with
multiplicity \(\binom{n}{k}\), and the
eigenspace \(V_{n-2k}\) 
from \eqref{eq:hypercube-eigenspaces},
on which \(S_n\) acts by permuting the set
\(\{u \in \ZZ_2^n:\omega(u)=k\}\),
or, equivalently, by permuting the \(k\)-subsets of \(\{1,2,\ldots,n\}\).
\end{proposition}

%%%%%%%%%%%%%%%%%%%%%%%%%%%%%%%%%%%%%%%%%%%%%%%%%
\section{Common eigenspaces and the main result}
\label{sec:common-eigensapces}
%%%%%%%%%%%%%%%%%%%%%%%%%%%%%%%%%%%%%%%%%%%%%%%%%

Consider the nested subspaces
\(\hat{\hat{V}} \subset \hat{V} \subset V\), where \(\hat{V}\) is the
subspace of \(V\) in which the \(\zero_n\)-coordinate vanishes, and
\(\hat{\hat{V}}\) is the subspace in which both the \(\zero_n\)- and
\(\one_n\)-coordinates vanish. The adjacency matrices
\(A_{\hat{Q}_n}\) and \(A_{\hat{\hat{Q}}_n}\) act naturally on
\(\hat{V}\) and \(\hat{\hat{V}}\), respectively. Specifically, as in
\eqref{eq:AQn}, for
\( u\in \ZZ_2^n \setminus \{\zero_n\} \) and
\( v\in \ZZ_2^n \setminus \{\zero_n, \one_n\} \), we have
\[
  A_{\hat{Q}_n} \ee_u = \sum_{\substack{w \in \ZZ_2^n \setminus \{\zero_n\}:\\ d(u,w)=1}} \ee_w, \qquad
  A_{\hat{\hat{Q}}_n} \ee_v = \sum_{\substack{w \in \ZZ_2^n \setminus \{\zero_n,\one_n\}:\\ d(v,w)=1}} \ee_w.
\]
Again, $S_n$ permutes the sets $\ZZ_2^n\setminus \{\zero_n\}$ and $\ZZ_2^n\setminus \{\zero_n,\one_n\}$, 
thus permuting the coordinates in $\hat{V}, \hat{\hat{V}}$.
The $S_n$-actions
are again via graph symmetries of $\hat{Q}_n, \hat{\hat{Q}}_n$, so that $S_n$ preserves
the eigenspaces of $A_{\hat{Q}_n}, A_{\hat{\hat{Q}}_n}$.

The next proposition
identifies large common eigenspaces for 
all three adjacency matrices.

\begin{proposition}
\label{prop:difference-eigenvectors}
Let $u,v$ be in $\ZZ_2^n$, 
with $u \neq v$ and $\omega(u)=\omega(v)=k$ for $1 \leq k \leq n-1$.
Then their difference vector $E_u-E_v$
lies in $\hat{\hat{V}}$,
and gives a common $(n-2k)$-eigenvector
for all three matrices $A_{\hat{\hat{Q}}_n}, A_{\hat{Q}_n}, A_{Q_n}$.

\end{proposition}
\begin{proof}
The assertion that the difference vector $E_u-E_v$
lies in $\hat{\hat{V}}$
follows from these equalities:
\begin{align*}
(E_u)_{\zero_n}&=(-1)^{u \cdot \zero_n}=(-1)^0=(-1)^{v \cdot \zero_n}=(E_v)_{\zero_n},\\
(E_u)_{\one_n}&=(-1)^{u \cdot \one_n}=(-1)^k=(-1)^{v \cdot \one_n} =(E_v)_{\one_n}.
\end{align*}
To prove the common eigenvector assertion, first note that given 
any $w$ in
$\ZZ_2^n$, one has
$$
(A_{Q_n} E_u)_w
=\sum_{\substack{w' \in \ZZ_2^n:\\
d(w,w')=1}} (E_u)_{w'}
=\sum_{\substack{w' \in \ZZ_2^n:\\
d(w,w')=1}} (-1)^{u \cdot w'}.
$$
Similarly, for \( w\in \ZZ_2^n \setminus \{\zero_n\} \), we have
\begin{align*}
(A_{\hat{Q}_n} (E_u-E_v))_w
&=\sum_{\substack{w' \in \ZZ_2^n \setminus \{\zero_n\}:\\
d(w,w')=1}}
[(-1)^{u \cdot w'}
- (-1)^{v \cdot w'}]
\overset{(a)}{=}\sum_{\substack{w' \in \ZZ_2^n:\\
d(w,w')=1}} 
[(-1)^{u \cdot w'}
- (-1)^{v \cdot w'}]\\
&=(A_{Q_n} (E_u-E_v))_w
\overset{(b)}{=}(n-2k) (E_u-E_v)_w,
\end{align*}
where equality (a) used the fact that $(-1)^{u \cdot \zero_n}=(-1)^{v \cdot \zero_n}$, and (b) used the eigenvector equation \eqref{eqn:eigenvector-equation-for-Eu}.  
Similarly, since $(-1)^{u \cdot \one_n}=(-1)^k=(-1)^{v \cdot \one_n}$, 
for any $w\in \ZZ_2^n \setminus \{\zero_n,\one_n\}$, we have
\begin{align*}
(A_{\hat{\hat{Q}}_n} (E_u-E_v))_w
&=\sum_{\substack{w' \in \ZZ_2^n \setminus \{\zero_n,\one_n\}:\\
d(w,w')=1}} 
[(-1)^{u \cdot w'}
- (-1)^{v \cdot w'}]
=\sum_{\substack{w' \in \ZZ_2^n:\\
d(w,w')=1}} 
[(-1)^{u \cdot w'}
- (-1)^{v \cdot w'}]\\
&=(A_{Q_n} (E_u-E_v))_w
=(n-2k) (E_u-E_v)_w. \qedhere
\end{align*}
\end{proof}

Bearing in mind Proposition~\ref{prop:difference-eigenvectors}, define these codimension one subspaces of
$V_{n-2k}$ for $k=0,1,\ldots,n$:
\begin{equation}
\label{eq:common-eigenspaces}
\hat{\hat{V}}_{n-2k}:=
\spn_\RR\{E_u-E_v: u,v \in \ZZ_2^n, \omega(u)=\omega(v)=k\}.
\end{equation}
Note that \( \hat{\hat{V}}_{n-2k}=\{0\} \) for \( k=0 \) and \( k=n \).
The spaces \( \hat{\hat{V}}_{n-2k} \) for \( 1\le k\le n-1 \) already account for most of the eigenspace decompositions of the three symmetric matrices $A_{Q_n}, A_{\hat{Q}_n}, A_{\hat{\hat{Q}}_n}$.
To find the complete descriptions of the eigenspace decompositions, we define
for $k=0,1,\ldots,n$ the vectors 
\[
  f_k:=\sum_{\substack{u \in \ZZ_2^n:\\\omega(u)=k}} E_u,
\]
which are fixed by $S_n$. We also define
\begin{align*}
U &:=\spn_\RR\{f_0,f_1,f_2,\ldots,f_{n-1},f_n\} \subseteq V,\\
\hat{U}&:= U\cap\hat{V},\\
\hat{\hat{U}}&:= U\cap\hat{\hat{V}}.
\end{align*}
By Proposition~\ref{prop:eigenspaces-for-hypercube}, 
one then has these orthogonal direct sum decompositions with respect to $(-,-)$:
%the standard inner product on $V$: for \( 1\le k\le n-1 \),
$$
V_{n-2k}=\hat{\hat{V}}_{n-2k}
\oplus \spn_\RR \{f_k\}
\quad \text{ for }1\le k\le n-1.
$$

We give a more concrete description, better suited to analyzing the actions of  
$A_{Q_n}, A_{\hat{Q}_n}, A_{\hat{\hat{Q}}_n}$ on $U, \hat{U},\hat{\hat{U}}$ respectively.
Define for $k=0,1,2,\ldots,n$ vectors
$e_k$ in $\RR^{\ZZ_2^n}$ by
$(e_k)_v=1$ if $\omega(v)=k$ and
$(e_k)_v=0$ otherwise.

\begin{proposition}
\label{prop:leftover-space-identified}
One has these $\RR$-bases for $\hat{\hat{U}},\hat{U},U$:
\begin{align*}
U&=\spn_\RR\{e_0,e_1,e_2,\ldots,e_{n-1},e_n\},\\
\hat{U}&=\spn_\RR\{e_1,e_2,\ldots,e_{n-1},e_n\},\\
\hat{\hat{U}}&=\spn_\RR\{e_1,e_2,\ldots,e_{n-1}\}.
\end{align*}
That is, $\hat{\hat{U}},\hat{U},U$ are the subspaces of $\hat{\hat{V}},\hat{V},V$ whose vectors $g$ have coordinates $g_v$ depending only on $\omega(v)$.
\end{proposition}
\begin{proof}
We first claim that $U, \hat{U}, \hat{\hat{U}}$
have dimensions $n+1,n,n-1$, respectively. This holds
for $U$ since its basis is $\{f_k\}_{k=0,1,\ldots,n}$.
It follows for $\hat{U}, \hat{\hat{U}}$ as they are the kernels of maps
\( \hat{\varphi}:U\to \RR \)
and 
\( \hat{\hat{\varphi}}:U\to \RR^2 
\) 
given by
  \( \hat{\varphi}(g)= (g)_{\zero_n} \) and
  \( \hat{\hat{\varphi}}(g) = ((g)_{\zero_n},(g)_{\one_n}) \), 
having ranks $1$ and $2$,
since 
  \( (f_k)_{\zero_n}=\binom{n}{k} \) and
  \( (f_k)_{\one_n}=(-1)^k\binom{n}{k} \).
% Note that $\dim U=n+1$. Let \( \hat{\varphi}:U\to \RR \) and \( \hat{\hat{\varphi}}:U\to \RR^2 \) be the linear maps given by \( \hat{\varphi}(g)= (g)_{\zero_n} \) and \( \hat{\hat{\varphi}}(g) = ((g)_{\zero_n},(g)_{\one_n}) \). Then  \( \hat{U} = \ker \hat{\varphi} \) and \( \hat{\hat{U}} = \ker \hat{\hat{\varphi}} \). Since \( (f_k)_{\zero_n}=\binom{n}{k} \) and \( (f_k)_{\one_n}=(-1)^k\binom{n}{k} \), the linear maps  \( \hat{\varphi} \) and \( \hat{\hat{\varphi}} \) have ranks \( 1 \) and \( 2 \), respectively. Thus, \( \dim\hat{U} = \dim \ker \hat{\varphi} = n \) and  \( \dim\hat{\hat{U}} = \dim \ker \hat{\hat{\varphi}} = n-1 \).
  
  Consequently, by dimension counting, it suffices to show that, for
  each $k$, the coordinate $(f_k)_v$ depends only on
  $\omega(v)=:\ell$. Letting
$$
(\one_\ell, \zero_{n-\ell}):=(\underbrace{1,1,\ldots,1}_{\ell\text{ entries}},
\underbrace{0,0,\ldots,0}_{n-\ell\text{ entries}})
$$
one has this calculation:
\begin{equation}
\label{eq:f-have-constant-coordiantes}
(f_k)_v 
=
\sum_{\substack{u \in \ZZ_2^n:\\\omega(u)=k}}
(E_u)_v 
=
\sum_{\substack{u \in \ZZ_2^n:\\\omega(u)=k}}
(-1)^{u \cdot v} 
\overset{(*)}{=}\sum_{\substack{u \in \ZZ_2^n:\\\omega(u)=k}}
(-1)^{u \cdot (\one_\ell,\zero_{n-\ell})}
\end{equation}
where \( (*) \) used the fact
that the summation set $\{u \in \ZZ_2^n: \omega(u)=k\}$ is stable under
$S_n$ permuting coordinates in $\ZZ_2^n$, and $v, (\one_\ell,\zero_{n-\ell})$ lie
 in the same $S_n$-orbit. 
So the coordinate $(f_k)_v$  in \eqref{eq:f-have-constant-coordiantes} depends only on $\ell=\omega(v)$.
\end{proof}

One can readily check that 
$
A_{Q_n}e_k=(k+1) e_{k+1}+(n-k+1)e_{k-1},
$
with conventions $e_{-1}=0=e_{n+1}$.  Therefore $A_{Q_n}$ acts on $U$ in the ordered basis 
$(e_0,e_1,\ldots,e_{n-1},e_n)$ 
via this tridiagonal matrix:
\begin{equation}
\label{eq:transposed-Kac-matrix}
K:=\quad
\left[\begin{smallmatrix}
 0 & n &    &       &       &  &    & \cr
 1 & 0 & n-1&       &       &   &   & \cr
  & 2 &  0 &n-2    &       &    &  & \cr
   &   & 3  &0      & &     & & \cr
    &   &    & &\ddots      &  & \cr
 &   &    & & &0  &2    &  & \cr
&   &    &   &    &n-1    &0     & 1\cr 
&    &    &   &    &       & n    & 0 
\end{smallmatrix}\right].
\end{equation}
Meanwhile $A_{\hat{Q}_n}, A_{\hat{\hat{Q}}_n} $ act on $\hat{U}, \hat{\hat{U}}$ via the principal submatrices $L, M$ inside $K$ 
of sizes $n, n-1$ obtained by removing the first row and column to give $L$, and
then further removing the last row and column to give $M$.  For the sake of stating our main result, define the three characteristic polynomials
\begin{align*}
    K(x)&:= \det(x 1_U - K), \\
    L(x)&:=\det(x 1_{\hat{U}} - L), \\
    M(x)&:=\det(x 1_{\hat{\hat{U}}} - M).
\end{align*}

\begin{theorem}
\label{thm:main}
The roots of $K(x), L(x), M(x)$
and the eigenspaces of $A_{Q_n}, A_{\hat{Q}_n},A_{\hat{\hat{Q}}_n}$ are related as follows.
\begin{itemize}

\item[(i)] 
All three polynomials $K(x), L(x), M(x)$
have only simple roots, which are symmetric about $0$ in $\RR$.

\item[(ii)]
The roots of $K(x)$ are the same as the eigenvalues of  $A_{Q_n}$, namely 
$
\{-n,-(n-2),\ldots,n-2,n\}.
$

\item[(iii)] The roots of $L(x)$ strictly interlace the roots of $K(x)$. 

\item[(iv)] The roots of $M(x)$ strictly interlace the roots of $L(x)$.

\item[(v)] Positive (resp. negative) roots of $M(x)$ strictly interlace the positive (resp. negative) roots of $K(x)$.

\item[(vi)] 
The matrices $A_{\hat{Q}_n},A_{\hat{\hat{Q}}_n}$ 
have the common $(n-2k)$-eigen-subspaces
$
\hat{\hat{V}}_{n-2k}
$
from \eqref{eq:common-eigenspaces}, for $1\leq k\leq n-1$,
whose $S_n$-action is the permutation action on $k$-subsets of $\{1,2,\ldots,n\}$ with the trivial representation removed.

\item[(vii)]
The matrices $A_{\hat{Q}_n},A_{\hat{\hat{Q}}_n}$ 
have additional $1$-dimensional eigen-subspaces with trivial $S_n$-actions, one for each root of the characteristic polynomials $L(x), M(x)$, respectively.  When $n$ is even, the root $0$ of $M(x)$ coincides with $n-2(n/2)$, so the $0$-eigenspace of $A_{\hat{\hat{Q}}_n}$ is the direct sum of $\hat{\hat{V}}_0$ and the $1$-dimensional zero-eigen-subspace arising from $M(x)$.

\end{itemize}
\end{theorem}

\noindent
Schematically, here is the placement of the roots as in (i)-(v) for
$M(x), L(x), K(x)$, indicated with $m,\ell,\colorize{k}$:

\vskip.1in
\noindent
{\sf For n even}:
\tiny
$$
\begin{matrix}
\colorize{-n} & & &\colorize{-(n-2)} & \cdots &\colorize{-6}& & &\colorize{-4}& & &\colorize{-2} & & & \colorize{0}  & & & \colorize{2} & &  & \colorize{4}& &  & \colorize{6}& \cdots &\colorize{n-2}& & &\colorize{n}\\ 
\colorize{k} &\ell & m& \colorize{k}&\cdots &\colorize{k}& \ell&m &\colorize{k} &\ell & m &\colorize{k} &\ell && \colorize{k} &&\ell & \colorize{k} &m &\ell & \colorize{k}&m&\ell& \colorize{k}& \cdots &\colorize{k}&m &\ell &\colorize{k}\\ 
 & & & & & & & & & & & & && m && &  & & & & & & &  && & & \\ 
\end{matrix}
$$
\normalsize
\noindent
{\sf For n odd}:
\tiny
$$
\begin{matrix}
\colorize{-n} & & &\colorize{-(n-2)} & \cdots &\colorize{-5}& & &\colorize{-3}& & &\colorize{-1} & 0 & \colorize{1}  & & & \colorize{3} & &  & \colorize{5}& \cdots &\colorize{n-2}& & &\colorize{n}\\
\colorize{k} &\ell & m& \colorize{k}&\cdots &\colorize{k}& \ell&m &\colorize{k} &\ell & m &\colorize{k} &\ell & \colorize{k} & m &\ell & \colorize{k} &m &\ell & \colorize{k}& \cdots &\colorize{k}&m &\ell &\colorize{k}\\ 
\end{matrix}
$$

\normalsize
\begin{proof}[Proof of Theorem~\ref{thm:main}.]
Assertions (ii), (vi), (vii) follow from the discussion surrounding
Propositions~\ref{prop:eigenspaces-for-hypercube},
\ref{prop:difference-eigenvectors}, \ref{prop:leftover-space-identified}.

Most of the remaining assertions
follow from general
theory of eigenvalues for
{\it tridiagonal matrices}, as in Meurant \cite{Meurant},
or roots of {\it orthogonal polynomials}, as in Ismail \cite{Ismail}.
For assertion (i), symmetry about $0$ for
the roots of $K(x), L(x),M(x)$ already follows either from
the fact that they are eigenvalues of adjacency matrices of {\it bipartite} graphs (see, e.g., \cite[Thm.~3.11]{CvetkovicDoobSachs}), or from the fact that they are eigenvalues
of tridiagonal matrices whose diagonals are all zero.  The
simplicity assertion in (i) and  the strict interlacing assertions (iii), (iv) for $L, K$ and for $M, L$ follow from a strict version of Cauchy's Interlacing Theorem: these are principal submatrices inside tridiagonal matrices having {\it positive subdiagonal and superdiagonal entries}; see, e.g.,  Meurant \cite[Prop.~2.21]{Meurant} or Ismail \cite[Thm.~2.2.3]{Ismail}. 

Assertion (v) is the most subtle.  One would like 
to prove that $M$ has its strictly positive eigenvalues interlacing
$K$'s strictly positive eigenvalues
$x_m=n-2m$ for $0 \leq m < \frac{n}{2}$.  Equivalently, defining
$\sign(\alpha):=\frac{\alpha}{|\alpha|}$
for $\alpha \in \RR \setminus \{0\}$ (and $\sign(0)=0$), it suffices to show the sign-alternation
\begin{equation}
\label{eq:the-desired-sign-alternation}
\sign(M(x_m))=(-1)^m \text{ for }0 \leq m < \frac{n}{2},
\end{equation}
as then $M(x)$ has a root between each consecutive pair of positive roots of \(K\).
A query to ChatGPT 5.5 Pro produced a proof of \eqref{eq:the-desired-sign-alternation}, using detailed info about $K$. We explain that proof in the next section.
\end{proof}

%%%%%%
\section{Proof of \eqref{eq:the-desired-sign-alternation}.}
\label{sec:ChatGPT-proof}
%%%%%
Rather than working with
eigenvalues and characteristic polynomials of $M,L,K$, we use the transpose,
\begin{equation}
\label{eq:Kac-matrix}
K^T:=\quad
\left[\begin{smallmatrix}
 0 & 1 &    &       &       &  &    & \cr
 n & 0 & 2&       &       &   &   & \cr
  & n-1 &  0 &3    &       &    &  & \cr
   &   & n-2  &0      & &     & & \cr
    &   &    & &\ddots      &  & \cr
 &   &    & & &0  &n-1    &  & \cr
&   &    &   &    &2    &0     & n\cr 
&    &    &   &    &       & 1    & 0 
\end{smallmatrix}\right].
\end{equation}
This matrix $K^T$ has been called the {\it Kac matrix} or {\it Kac--Clement--Sylvester matrix}, referring to Sylvester \cite{Sylvester}, Kac \cite{Kac}, and Clement \cite{Clement}.  Its eigenvalues, eigenvectors, and history are discussed extensively by Taussky and Todd \cite{TausskyTodd}.  
Kac observed that the differential operator  
$
\mathcal{K}:=
nt+(1-t^2)\frac{d}{dt}
$
satisfies $\mathcal{K} t^k = k t^{k-1}+(n-k)t^{k+1}$, so it
acts via $K^T$ on the $\RR$-basis $\{1,t,t^2,\ldots,t^n\}$ for
polynomials inside $\RR[t]$ of degree at most $n$, with the polynomials
\begin{equation}
\label{eq:Kac-polynomials}
f_{n,m}(t)=(1-t)^m (1+t)^{n-m}
\end{equation}
forming
a $\mathcal{K}$ eigenbasis:
\begin{equation}
\label{eq:Kac-eigenfunction-equation}
\mathcal{K} \cdot f_{n,m}(t)= x_m f_{n,m}(t),
\text{ where }x_{m}=n-2m.
\end{equation}

By the definitions of $K,L,M$ and their symmetry
and tridiagonal sparsity, applying the {\it Dodgson condensation} or {\it Desnanot--Jacobi identity}
to $x I_{n+1}-K$ (see Bressoud \cite[\S 3.5]{Bressoud}) gives
$$
K(x) M(x) = L(x)^2 - (n!)^2.
$$
Differentiating this equation with respect to $x$ yields
$$
K'(x) M(x) + K(x) M'(x) = 2 L(x) L'(x),
$$
and then plugging in any root $x_m=n-2m$ 
of $K(x)$ gives
$$
K'(x_m) M(x_m) = 2L(x_m) L'(x_m).
$$
Consequently, for each $m=0,1,\ldots, n$, one has
\begin{equation}
\label{eq:sign-relations}
 \sign K'(x_m) \cdot \sign M(x_m) =\sign L(x_m) \cdot \sign L'(x_m).
\end{equation}
As $\{x_m\}_{m=0,1,\ldots,n}$ are the roots of the monic polynomial
$K(x)$, all simple roots, one has $\sign K'(x_m)=(-1)^m$. Since the
roots of \( K(x) \) and \( L(x) \) strictly interlace, we also have
\( \sign L(x_m)=(-1)^m \). Hence, for the desired sign alternation
\eqref{eq:the-desired-sign-alternation}, it remains to prove that
\( \sign L'(x_m)=(-1)^m \). To this end, we find an explicit formula
for \( L'(x_m) \). First, we show that \( L(x_m) \) has the following
simple formula, which also implies \( \sign L(x_m)=(-1)^m \).

\begin{lemma}
\label{lem:L-at-root-sign-calculations}
For $m=0,1,\ldots,n$, one has
$
L(x_m)=(-1)^m \cdot n!.
$
\end{lemma}

\begin{proof}
We prove a more general claim.  Let $L_k(x):=\det(x I _k -L_k)$ 
where $L_k$ is the upper left $k \times k$ principal submatrix of $K$, so that
$$
L_0(x)=1, \,\,
L_1(x)=x, \,\,
L_2(x)=x^2-n,\,\, \ldots, \,\,
L_n(x)=L(x).
$$
These polynomials satisfy the recurrence obtained by Laplace expansion:
\begin{equation}
\label{eq:Laplace-expansion-recurrence}
L_{k+1}(x)= x \cdot L_k(x)-k(n-k+1) \cdot L_{k-1}(x).
\end{equation}
{\bf Claim:} One has $L_k(x_m)=k!  c_{m,k}$
where $c_{m,k}$ are the expansion coefficients here:
$$
f_{n,m}(t)=(1-t)^m (1+t)^{n-m} =\sum_{k=0}^n
c_{m,k} t^k.
$$
Note that at $k=n$, this claim says
$L(x_m)=L_n(x_m)$ is $n!$ times the coefficient of $t^n$ in
$(1-t)^m (1+t)^{n-m}$, that is, $(-1)^m n!$,
proving Lemma~\ref{lem:L-at-root-sign-calculations}.  To prove the claim, note 
evaluating \eqref{eq:Laplace-expansion-recurrence} at $x=x_m$ 
gives this recurrence:
\begin{equation}
\label{eq:Laplace-expansion-at-roots}
L_{k+1}(x_m)= x_m \cdot L_k(x_m)-k(n-k+1) \cdot L_{k-1}(x_m).
\end{equation}
Meanwhile, extracting the coefficient of $t^k$ on both sides of the eigenvalue equation \eqref{eq:Kac-eigenfunction-equation} gives 
\begin{equation}
\label{eq:coefficient-extraction-equation}
(k+1)c_{m,k+1}+(n-k+1)c_{m,k-1}=x_m c_{m,k}.
\end{equation}
Multiplying \eqref{eq:coefficient-extraction-equation} by $k!$ and rewriting gives this recurrence:
\begin{equation}
\label{eq:rewritten-eigenvalue-equation}
(k+1)! c_{m,k+1}=x_m \cdot k! c_{m,k}- k(n-k+1)\cdot (k-1)! c_{m,k-1}.
\end{equation}
Comparing recurrences \eqref{eq:Laplace-expansion-at-roots}, \eqref{eq:rewritten-eigenvalue-equation} proves the claim by induction on $k$,
with base cases $k=0,1$ immediate.
\end{proof}

We now establish a formula for \( L'(x_m) \).

\begin{lemma}
\label{lem:Lprime-at-root-sign-calculations}
For $0\leq m<n/2$, one has
\begin{equation}
\label{eq:Lprime-factorial-sum}
 L'(x_m)=(-1)^m2^{n-1}
 \sum_{j=m}^{n-m-1}j!(n-1-j)!.
\end{equation}
\end{lemma}

\begin{proof}
  By Lemma~\ref{lem:L-at-root-sign-calculations}, we have
  \(L(n-2m)=(-1)^m n!\) for \(m=0,1,\ldots,n\). 
  Thus 
  $L(n-2y) + L(n-2y-2)$
  vanishes at \( y=0,1,\ldots,n-1\). Since
  this is a polynomial in $y$ of degree $n$  with leading
  coefficient \( 2(-2)^n \) (as $L(x)$ was a monic polynomial in $x$ of degree $n$), we obtain
\begin{equation}
\label{eq:P-finite-difference}
L(n-2y) + L(n-2y-2)
 =2(-2)^n\prod_{j=0}^{n-1}(y-j).
\end{equation}
Differentiating this with respect to $y$ and then setting \(y=m\), where \(0\leq m\leq n-1\), gives
\begin{equation}
\label{eq:Lprime-adjacent-relation}
 L'(x_m)+L'(x_{m+1})
 =(-1)^m2^n m!(n-m-1)!.
\end{equation}
The recurrence \eqref{eq:Laplace-expansion-recurrence} for \(L_k\) gives \(L(-x)=(-1)^nL(x)\), and therefore
\(L'(-x)=(-1)^{n+1}L'(x)\).  Since \(x_{n-m}=-x_m\), it follows that
\begin{equation}\label{eq:1}
 L'(x_{n-m})=(-1)^{n+1}L'(x_m).
\end{equation}
For \(0\leq m<n/2\), multiply
\eqref{eq:Lprime-adjacent-relation}, with \(m\) replaced by \(j\), by
\((-1)^j\) and sum over \(j=m,m+1,\ldots,n-m-1\).  The left-hand side
telescopes, giving
\[
 (-1)^mL'(x_m)-(-1)^{n-m}L'(x_{n-m})
 =2^n\sum_{j=m}^{n-m-1}j!(n-1-j)!.
\]
By \eqref{eq:1}, the left-hand side is
\(2(-1)^mL'(x_m)\), which proves \eqref{eq:Lprime-factorial-sum}.
\end{proof}

By Lemma~\ref{lem:Lprime-at-root-sign-calculations}, we have
$\sign L'(x_m)=(-1)^m$, which completes the proof of
\eqref{eq:the-desired-sign-alternation}.

\begin{remark} \rm
    One might wonder whether the most subtle interlacing assertion (v) in Theorem~\ref{thm:main} holds more generally for tridiagonal matrices like the Kac matrix, having zero diagonal entries and positive entries on the superdiagonal and subdiagonal, perhaps under similar symmetry conditions. However, small examples indicate that additional inequalities on the entries are still required.  For example, letting
$$
    A=\left[ \begin{smallmatrix}
        0 & a &  & \\
        c & \colorize{0} &\colorize{b} & \\
         & \colorize{b} & \colorize{0} & c\\
         &  & a & 0
\end{smallmatrix}\right]
\quad \text{ and }
B=\left[ \colorize {\begin{smallmatrix}
        0 & b \\
        b & 0 
\end{smallmatrix}}\right]
$$
with $a,b,c>0$, a direct calculation shows that the unique positive root $b$ of $\det(xI_2-B)=(x-b)(x+b)$ 
lies strictly between the
two positive roots of 
$\det(xI_4-A)=
(x^2-bx-ac)(x^2+bx-ac)$
if and only if $b > \sqrt{ac/2}$.
\end{remark}

\begin{remark} \rm
The polynomials \( L_k(x) \) in the proof of
Lemma~\ref{lem:L-at-root-sign-calculations} are related to the
\emph{Krawtchouk polynomials} \( K_k(x;p,n) \), where \( n \) is a
fixed integer, \( p \) is a nonzero real number, and \( 0\le k\le n \). Their generating function is
\begin{equation}
\label{eq:general-Krawtchouk-generating-function}
 \left(1-\frac{1-p}{p}t\right)^x(1+t)^{n-x}
 =\sum_{k=0}^n\binom{n}{k} K_k(x;p,n)t^k.
\end{equation}
See
\cite[eqn.~(9.11.11)]{KLS} or \cite[eqn.~(18.23.3)]{NIST}.
Hence, by the claim in the proof of
Lemma~\ref{lem:L-at-root-sign-calculations}, we have
\[
 L_k(n-2m)=\frac{n!}{(n-k)!}
 K_k\left(m;\frac12,n\right)
 \qquad(0\leq k\leq n).
\]
\end{remark}

\begin{remark} \rm
\label{rem:combinatorial-proof-of-lemma}
As the formula \(L(x_m)=(-1)^m\cdot n!\) in
Lemma~\ref{lem:L-at-root-sign-calculations} is so simple, we give here
a combinatorial proof.

\begin{proof}[Combinatorial proof of Lemma~\ref{lem:L-at-root-sign-calculations}]
Fix an alphabet $\mathcal B=\{1,2,\ldots,n\}$ in which the first $m$
letters have sign $-1$ and the remaining $n-m$ letters have sign $+1$.
Write this sign as $\epsilon(a)$, so that
$\sum_{a\in\mathcal B}\epsilon(a)=n-2m=x_m.$
For $0\leq k\leq n$, let $\mathcal W_k$ be the set of injective words
$w=a_1a_2\cdots a_k$ over $\mathcal B$, and give each such word the product weight
$
 \operatorname{wt}(w):=\prod_{i=1}^k\epsilon(a_i).
$
Letting  $A_k:=\sum_{w\in\mathcal W_k}\operatorname{wt}(w)$,
one has
$A_0=1$ and $A_1=x_m$, which 
will give the base cases in an inductive strategy to show 
\begin{equation}
\label{eq:sign-weighted-sum-gives-Lk-evaluations}
A_k=L_k(x_m) \text{ for }k=0,1,\ldots,n.
\end{equation}
Note that \eqref{eq:sign-weighted-sum-gives-Lk-evaluations} at $k=n$ proves the lemma: every word in $\mathcal W_n$ is a permutation of
$\mathcal B$, and each of these
$n!$ words has weight
$(-1)^m$, giving
$
 L(x_m)=L_n(x_m)=A_n=(-1)^m n!.
$
To prove \eqref{eq:sign-weighted-sum-gives-Lk-evaluations},
recall the recurrence \eqref{eq:Laplace-expansion-at-roots} 
\[
 L_{k+1}(x_m)
 =x_m \, L_k(x_m)-k(n-k+1) \, L_{k-1}(x_m).
\]
It suffices to show $A_k$ satisfies the same
recurrence
$A_{k+1}=x_m\, A_k-k(n-k+1) \, A_{k-1}$. Rewrite this as
\begin{equation}
\label{eq:rewritten-word-recurrence}
 x_m\, A_k=A_{k+1}+k(n-k+1) \, A_{k-1}.
\end{equation} 
Prove \eqref{eq:rewritten-word-recurrence} by writing $x_m\, A_k=\sum_{(w,a) \in \mathcal W_k \times \mathcal B} \operatorname{wt}(w)\epsilon(a)$, and decomposing based on whether $a$ occurs in $w$.
If $a$ does
not occur in $w$, appending $a$ gives a word in ${\mathcal W}_{k+1}$, and such pairs $(w,a)$ have total weight $A_{k+1}$.  
If $a$ occurs
in $w$, deleting that occurrence gives a word
$v\in\mathcal W_{k-1}$, together with the deleted position, which has
$k$ choices, and a letter $a$ not occurring in $v$, which has
$n-k+1$ choices.  Moreover,
$
 \operatorname{wt}(w)\epsilon(a)
 =\operatorname{wt}(v)\epsilon(a)^2
 =\operatorname{wt}(v).
$
Thus pairs with $a$ in $w$ have total weight $k(n-k+1)A_{k-1}$, proving \eqref{eq:rewritten-word-recurrence}.
\end{proof}
\end{remark}

%%%%
\section*{AI Disclosure}
The proof of part (v) of Theorem~\ref{thm:main} and the combinatorial
proof of Lemma~\ref{lem:L-at-root-sign-calculations} given in
Remark~\ref{rem:combinatorial-proof-of-lemma} were obtained through
several queries to {\tt ChatGPT 5.5 Pro} and {\tt ChatGPT 5.6 Pro}.
The authors subsequently rewrote both proofs and take full responsibility for their correctness.

%%%%

%%%%
\section*{Acknowledgements}
The authors thank Laura Escobar and Dennis Stanton for helpful conversations.
The first author was supported by the National Research Foundation of Korea (NRF) grant funded by the Korea government RS-2025-00557835.
The second author was supported by NSF grant DMS-2450430.
%%%%

\bibliographystyle{abbrv}
\bibliography{references.bib}

\end{document}